\RequirePackage{ifpdf}
\ifpdf 
\documentclass[pdftex]{sigma}
\else
\documentclass{sigma}
\fi

\newcommand{\ds}{\displaystyle}

\newcommand{\Ess}{\emph{Ess}}

\begin{document}

\renewcommand{\PaperNumber}{***}

\FirstPageHeading

\ShortArticleName{On a class of Leibniz Algebras}

\ArticleName{ON A CLASS OF LEIBNIZ ALGEBRAS}
\Author{C\^ome J.A.B\'ER\'E~$^\dag$,  Aslao KOBMBAYE~$^\ddag$ and Amidou KONKOBO~$^\dag$}

\AuthorNameForHeading{C\^ome B\'er\'e et al.}

\Address{$^\dag$~\footnotesize \it Laboratory T.N. AGATA \\Department of Mathematics and Computer Science\\ %
University of Ouagadougou Burkina Faso, Address 03 B.P. 7021 Ouagadougou 03 Burkina Faso.}
\EmailD{\href{mailto:bere_jean0@yahoo.fr}{come\_bere@univ-ouaga.bf \& amidoukonkobo@univ-ouaga.bf}} %

\Address{$^\ddag$~\footnotesize \it Laboratory T.N. AGATA  \\Department of Mathematics\\ 
University of Djamena Tchad, Address B.P. 1027 Djamena Tchad.}
\EmailD{\href{mailto:asljoskob@yahoo.fr}{asljoskob@yahoo.fr}} 

\ArticleDates{Received ???, in final form ????; Published online ????}

\Abstract{We pointed out the class of Leibniz algebras such that the Killing form is non degenerate implies algebra is semisimple.}


\Classification{17A32; 17B30.} %

\section{Introduction}

Throughout this paper, $F$ will be an algebraically closed field
of characteristic zero. All vector spaces and algebras will be finite
dimensional over $F$.   
Note the sum of two vector subspaces
$V_1,V_2$ by $V_1\dot{+}V_2$ and direct sum by $V_1\oplus V_2$.
It  is well-known that  a Lie algebra is semisimple if and only if its Killing form is non degenerate. An equivalent criterion is found 
for Leibniz algebra $L$ which satisfies, for all $x, y$ in $L$,  the trace of endomorphism $(ad_x\circ ad_y)_{|\Ess(L)}$ equals zero. 
Call such algebras "Killing- Leibniz-Algebra".

Section \ref{sec2} is devoted to basic facts. 
In Section \ref{sec3},    
the links between radical and nilradical are set. %
Section \ref{sec4} is devoted to the nilpotency of the ideal $\left\{Rad(L),L\right\}$.
In Section \ref{sec5}, the main theorem is settled.  
For conclusion, we give an hierarchy of Leibniz algebras  and two questions 
%
are  done about Killing Leibniz Algebras. 

\section{Basics facts.}\label{sec2}

Let us note that Leibniz algebras are defined in two classes:\\
\begin{itemize}
 \item Right Leibniz algebras, with the rule 
\begin{equation}\label{jacobi}
[x,[y,z]]=[[x,y],z]-[[x,z],y]\textnormal{ for any }x,y,z\in L.
\end{equation}
 \item Left Leibniz algebras, with the rule 
\begin{equation} 
[x,[y,z]]=[[x,y],z]+[y,[x,z]]\textnormal{ for any }x,y,z\in L.
\end{equation}
\end{itemize}
For an algebra $(A, [\,,\,])$ with vectors multiplication $[a,b]$, for all $a$, $b$ in $A$, define the algebra $(A, [\,,\,]^{op})$  as the underlying vector space $A$ where the  vectors multiplication is defined by $[a,b]^{op}=[b,a]$.
We have that:\\
\begin{proposition}
The algebra $(A, [\,,\,])$ is left Leibniz algebra if and only if the algebra $(A, [\,,\,]^{op})$ is right Leibniz algebra.
\end{proposition} 
So results on Left Leibniz algebras are available on Right Leibniz algebras, (with minors variations). 

Here we write "Leibniz algebras" for "Right Leibniz algebras".

\medskip{}
It follows from the equation (\ref{jacobi}) called Leibniz identity that in any Leibniz algebra one
has 
\[
[y,[x,x]]=0,\,[z,[x,y]]+[z,[y,x]]=0,\textnormal{ for all }x,y,z\in L.
\]
\begin{definition} {(Ideal)}\, 
A subspace $H$ of a Leibniz algebra $L$ is called left (respectively
right) ideal if for $a\in H$ and $x\in L$ one has $[x,a]\in H$
(respectively $[a,x]\in H$). If $H$ is both left and right ideal,
then $H$ is called (two-sided) ideal.

If $V$ is a vector space, let $End_F(V)$ denotes the set of all
endomorphisms of $V$. An action of $L$ on $End_F(V)$ is a linear map 
of $L$ on $End_F(V)$.
\end{definition}


\begin{definition}{(Representation)}\, 
Let $L$ be a Leibniz algebra and $V$ a vector space. $V$ is an $L$-module if
there are: 
\begin{itemize}
\item a left action, $l:L\longrightarrow End_F(V),\, x\mapsto l_x$ 
\item a right action, $r:L\longrightarrow End_F(V),\, x\mapsto r{}_{x}$,\\
such that: 
\end{itemize}
\[
\begin{array}{ccl}
r_{[x,y]} & = & r_yr_x-r_xr_y,\\
l_{[x,y]} & = & r_yl_x-l_xr_y,\\
l_{[x,y]} & = & r_yl_x+l_xl_y,
\end{array}
\]
 \label{def13}\end{definition}

For $x$ in $L$, $r_x(v)$ will be denoted by $vx$ and $l_x(v)$
will be denoted by $xv$. The triplet $(l,r,V)$ is called a representation
of $L$ on $V$. 
Now if $L$ is a Leibniz algebra, we have the adjoint representation
``$(Ad,ad,L)$'' defined as follows: 
for all $x$ and $y$ in $L$, $ad_{x}:L\longrightarrow L$, $y\longmapsto[y,x]$
and $Ad_{x}:L\longrightarrow L$, $y\longmapsto[x,y]$

\begin{remark}~ 
For $x\in L$, $ad_{x}:L\longrightarrow L$ is a derivation of $L$ i.e.
for all $x,y,z$ in $L$,\\  
$ad_{x}([y,z])=[ad_{x}(y),z]+[y,ad_{x}(z)]$.\\
 For $x\in L$, $Ad_{x}:L\longrightarrow L$ is an anti-derivation of $L$
i.e. for all $x,y,z$ in $L$,\\ 
$Ad_{x}([y,z])=[Ad_{x}(y),z]-[Ad_{x}(z),y]$.

\label{rmq2}\end{remark}

For an arbitrary algebra and for all non negative integer $n$ let
us define the sequences:
\begin{description} 
\item[(i)] $D^{1}\left(L\right)=L^{[1]}=L^{2}$, $D^{n+1}\left(L\right)=L^{[n+1]}=[L^{[n]},L^{[n]}]$;
\item[(ii)] $L^{1}=L$, $L^{n+1}=[L^{1},L^{n}]+[L^{2},L^{n-1}]+\cdots+[L^{n-1},L^{2}]+[L^{n},L^{1}]$.
\end{description}

\begin{definition}
(\cite{Albeverio+})\par 
An algebra $L$ is called solvable if there exists $m\in\mathbb{N^{*}}$
such that $D^{m}\left(L\right)=L^{[m]}=\{0\}$.\\
An algebra $L$ is called nilpotent if there exists $m\in\mathbb{N^{*}}$
such that $L^{m}=\{0\}$.
\end{definition}

\begin{definition}
Let $A$ be a subspace of a Leibniz algebra $L$. 
The normalizer of $A$ is denoted by : 
\[
n_{L}(A)=\left\{ y\in L,\left[y,a\right]\in A\textnormal{ and }\left[a,y\right]\in A\right\}. 
\] 
\end{definition} %

\begin{definition}
(\cite{demirs})\par 
A Leibniz algebra $L$ is said to be semisimple if $Rad(L) =\Ess(L)$.
\end{definition} %

Equivalently, we can say that :

  Leibniz algebra $L$ semisimple if 
 $\{0\}\neq[L,L]\neq\Ess(L)$ and every ideal  of $L$ belongs to the set $\left\{
L,\Ess(L),(0)\right\}$. 

Since $D\imath=\imath^{2}$ is an ideal whenever $\imath$ is (by Equation \ref{jacobi}
), if $rad\left(L\right)\neq\Ess(L)$ then $L$ contains
an ideal $\jmath$ which satisfies $\jmath^{2}\subseteq\Ess(L)\subsetneq\jmath$. 

So an other equivalent definition is: 
\begin{remark}\label{ari} $L$ is semisimple if it has no ideal $\jmath$
which satisfies $\jmath^{2}\subseteq\Ess(L)\subsetneq\jmath$. 
\label{remq2}\end{remark}


\begin{lemma}\label{Solvabl}\cite{bere12} 
Let $L$ be a Leibniz algebra and $\left(l,r,V\right)$
a representation of $L$. Let $A$ be a subspace of $L$, then $r_{A}=\left\{ r_x,\,\textnormal{for\,\ all}\,\ x\in A\right\} $
is a subspace of the vector space $End_{F}\left(V\right)$. In particular,
$r_{L}$ is a Lie subalgebra of $gl\left(V\right)$ and 
 $L$ is solvable (respectively nilpotent) if and only if $r_{L}$
is solvable (respectively nilpotent).
\label{lemma11}
\end{lemma}

\begin{proof} The results are clear since for all $x$, $y$ in $L$  and for all $\lambda$ in $F$,
we have that\\ $r_{x+\lambda y}=r_x+\lambda r_{y}$ and $\left[r_x,r_{y}\right]=r_{\left[y,x\right]}$.
 \end{proof}

\begin{remark}

Let $L$ be a Leibniz algebra and $\left(l,r,V\right)$
a representation of $L$.
If  for all $x$ in $L$, $r_x$ is nilpotent then $l_x$ is also nilpotent for all $x$. Since we have 
$l_x^k=(-1)^{k+1}l_x\left(r_x\right)^{k-1}$. Thus when  $r_x$ is nilpotent  for all $x$ in $L$, we can say that the representation $\left(l,r,V\right)$ of $L$ is nilpotent.
\end{remark}


\begin{lemma} 
Let $L$ be a Leibniz algebra and $\left(l,r,V\right)$ a representation
of $L$. Let $A$ be a subspace of the vector space $L$ and let $x$
in the normalizer $n_{L}(A)$ of $A$. Then we have for all integer
$k$ in $\mathbb{N}$ and for all $a$ in $A$:
\begin{description}
 \item[i) ] $\delta_{k+1}=r_{a}^{k+1}r_{x}-r_{x}r_{a}^{k+1}\in 
 r_A^{k+1}$.
 \item[ii)] $\beta_{k+1}=r_{x}^{k+1}r_{a}-r_{a}r_{x}^{k+1}\in 
 r_{A}r_{x}^{k}\dot{+}\cdots\dot{+}r_{A}r_{x}\dot{+}r_{A}$.
\end{description}
\end{lemma}

\begin{proof} For i), since 
  $\left[r_{a},r_{x}\right]=r_{\left[x,a\right]}$,  we have 
 $\delta_{1}=r_{a}r_{x}-r_{x}r_{a}=r_{\left[x,a\right]}$.
Thus $\delta_{1}\in r_{A}$ since $x\in n_{L}(A)$. 
And we have:
\[
\begin{array}{lll} 
\delta_{2} & = & r_{a}^{2}r_{x}-r_{x}r_{a}^{2}=r_{a}\left(r_{a}r_{x}\right)-r_{x}r_{a}^{2}\\
	   & = & r_{a}\left(r_{x}r_{a}+\delta_{1}\right)-r_{x}r_{a}^{2}= \left(r_{a}r_{x}\right)r_{a}+r_{a}\delta_{1}-r_{x}r_{a}^{2}\\
	   & = & \left(r_{x}r_{a}+\delta_{1}\right)r_{a}+r_{a}\delta_{1}-r_{x}r_{a}^{2}=\delta_{1}r_{a}+r_{a}\delta_{1}\\
	   & \in &  r_A^{2}.
\end{array}\]
With the hypothesis of recurrence: 
 $\delta_{k}=r_{a}^{k}r_{x}-r_{x}r_{a}^{k}\in r_A^{k}$,
we get:
\[\begin{array}{lll} 
\delta_{k+1} & = & r_{a}^{k+1}r_{x}-r_{x}r_{a}^{k+1}= r_{a}\left(r_{a}^{k}r_{x}\right)-r_{x}r_{a}^{k+1}\\
	     & = & r_{a}\left(r_{x}r_{a}^{k}+\delta_{k}\right)-r_{x}r_{a}^{k+1}=\left(r_{a}r_{x}\right)r_{a}^{k}+r_{a}\delta_{k}-r_{x}r_{a}^{k+1}\\
	     & = & \left(r_{x}r_{a}+\delta_{1}\right)r_{a}^{k}+r_{a}\delta_{k}-		r_{x}r_{a}^{k+1}= \delta_{1}r_{a}^{k}+r_{a}\delta_{k}\\
	     & \in & \left(r_{A}\right)^{k+1}.
\end{array}\]
 And for ii), 
 we have $\left[r_{x},r_{a}\right]=r_{\left[a,x\right]}$, so 
 $\beta_{1}=-\delta_1\in r_{A}=r_{A}r_{x}^{0}$
since $x\in n_{L}(A)$ (where $r_{x}^{0}=1_{V}$).
Note that we have:
\[\begin{array}{lll} 
\beta_{2} & = & r_{x}^{2}r_{a}-r_{a}r_{x}^{2}=r_{x}\left(r_{x}r_{a}\right)-r_{a}r_{x}^{2}\\
	  & = & r_{x}\left(r_{a}r_{x}+r_{\left[a,x\right]}\right)-r_{a}r_{x}^{2}=\left(r_{x}r_{a}\right)r_{x}+r_{x}r_{\left[a,x\right]}-r_{a}r_{x}^{2}\\
	  & = & \left(r_{a}r_{x}+r_{\left[a,x\right]}\right)r_{x}+\left(r_{\left[a,x\right]}r_{x}+r_{\left[\left[a,x\right],x\right]}\right)-r_{a}r_{x}^{2}=2r_{\left[a,x\right]}r_{x}+r_{\left[\left[a,x\right],x\right]}\\
	  & \in & r_{A}r_{x}\dot{+}r_{A}
\end{array}.\]
Set  
 $\beta_{k}=r_{x}^{k}r_{a}-r_{a}r_{x}^{k}\in r_{A}r_{x}^{k-1}\dot{+}\cdots\dot{+}r_{A}r_{x}\dot{+}r_{A}$,
and then it will follow that: 
\[\begin{array}{lll}
\beta_{k+1}&=&r_{x}^{k+1}r_{a}-r_{a}r_{x}^{k+1}=r_{x}^{k}\left(r_{x}r_{a}\right)-r_{a}r_{x}^{k+1}\\
&=&r_{x}^{k}\left(r_{a}r_{x}+r_{\left[a,x\right]}\right)-r_{a}r_{x}^{k+1}\\
&=&\left(r_{x}^{k}r_{a}\right)r_{x}+r_{x}^{k}r_{\left[a,x\right]}-r_{a}r_{x}^{k+1}\\
&=&\left(r_{a}r_{x}^{k}+\beta_{k}\right)r_{x}+r_{\left[a,x\right]}r_{x}^{k}\\
& &\hspace{2.3cm}+\beta_{1}^{\prime}-r_{a}r_{x}^{k+1} (\textnormal{ where }\beta_{1}^{\prime}=r_{x}^{k}r_{\left[a,x\right]}-r_{\left[a,x\right]}r_{x}^{k}=r_{x}^{k}%
\in r_A^{k})\\
&=&\beta_{k}r_{x}+r_{\left[a,x\right]}r_{x}^{k}+\beta_{1}^{\prime}\\
&\in&\left(r_{A}r_{x}^{k-1}\dot{+}\cdots\dot{+}r_{A}\right)r_{x}+r_Ar_{x}^{k}+\dot{+}r_{A}\\
&\in& r_{A}r_{x}^{k}\dot{+}r_{A}r_{x}^{k-1}\dot{+}\cdots\dot{+}r_{A}r_{x}\dot{+}r_{A}.
\end{array}.\]
Proofs are done.
\end{proof} 

\begin{lemma} 
Let $L$ be a Leibniz algebra and $\left(l,r,V\right)$ a representation
of $L$. Let $A$ be a subspace of the vector space $L$ and $x$
in the normalizer $n_{L}(A)$ of $A$. Then we have for all integer
$k$ and $p$ in $\mathbb{N}$:
\[\left[ r_A^{p}r_{x}^{k}\right]\circ  r_{A}\subseteq r_A^{p+1}r_{x}^{k}\dot{+}\cdots\dot{+} r_A^{p+1}r_{x}\dot{+} r_A^{p+1}.\]
\end{lemma}

\begin{proof}We shall note that:

$\begin{array}{ccl}
 \left[ r_A^{p}r_{x}^{k}\right]\circ  r_{A}& = & r_A^{p}\circ\left[r_{x}^{k}\circ r_{A}\right]\\ 
 & \subseteq &  r_A^{p}\left(r_{A}r_{x}^{k}\dot{+}\cdots\dot{+}r_{A}r_{x}\dot{+}r_{A}\right)\\
 & \subseteq &  r_A^{p+1}r_{x}^{k}\dot{+}\cdots\dot{+} r_A^{p+1}r_{x}\dot{+}\left(r_{A}%
\right)^{p+1}.
\end{array}$
\end{proof}

Thanks to the preceding lemma whe have for  all integer
$k,l,p$ and $q$ in $\mathbb{N}$:
\[ r_A^{p}r_{x}^{k}\circ r_A^{q}r_{x}^{l}\subseteq r_A^{p+q}r_{x}^{k+l}\dot{+}\cdots\dot{+} r_A^{p+q}r_{x}^{l}.\]


\begin{lemma} 
Let $L$ be a Leibniz algebra and $\left(l,r,V\right)$
a representation of $L$. Let $A$ be a subspace of the vector space
$L$ and $x$ in the normalizer $n_{L}(A)$ of $A$ and for a non negative integer $k$ let $E_k$ be the subspace $E_k=r_{A}r_{x}^{k}\dot{+}\cdots\dot{+}r_{A}$.  Then we have
for all integer $p$ in $\mathbb{N^*}$: 
\[
  E_k^{p}\subseteq%
 r_A^{p}r_{x}^{pk}\dot{+}\cdots\dot{+} r_A^{p}r_{x}^{2k}\dot{+}\cdots
\dot{+} r_A^{p}r_{x}\dot{+} r_A^{p}
\]
\end{lemma}

\begin{proof} Let us compute $E_k^p$ for $p=2,3$;  %
 we have $\left[r_{x},r_{a}\right]=r_{\left[a,x\right]}$, so \\
$\begin{array}{lll}
  E_k^2&=&\left(r_{A}r_{x}^{k}\dot{+}\cdots\dot{+}r_{A}\right)^{2}\\
 & = & \left(r_{A}r_{x}^{k}\dot{+}\cdots\dot{+}r_{A}\right)\left(r_{A}r_{x}^{k}\dot{+}\cdots\dot{+}r_{A}\right)\\
 & \subseteq & \left(r_{A}r_{x}^{k}\right)\left(r_{A}r_{x}^{k}\right)\dot{+}\cdots\dot{+}r_{A}\left(r_{A}r_{x}\right)\dot{+}\left(r_{A}r_{x}\right)r_{A}\dot{+}r_{A}r_{A}\\
 & \subseteq &  r_A^{2}r_{x}^{2k}\dot{+}\cdots\dot{+} r_A^{2}r_{x}^{k}\dot{+}\cdots\dot{+} r_A^{2}r_{x}\dot{+} r_A^{2}
\end{array}$\\
$\begin{array}{lll}
E_k^3&=&\left(r_{A}r_{x}^{k}\dot{+}\cdots\dot{+}r_{A}\right)^{3}\\
 & = & \left(r_{A}r_{x}^{k}\dot{+}\cdots\dot{+}r_{A}\right)^{2}\left(r_{A}r_{x}^{k}\dot{+}\cdots\dot{+}r_{A}\right)\\
 & \subseteq & \left( r_A^{2}r_{x}^{2k}\dot{+}\cdots\dot{+} r_A^{2}r_{x}^{k}\dot{+}\cdots%
\dot{+} r_A^{2}\right)\left(r_{A}r_{x}^{k}\dot{+}\cdots\dot{+}r_{A}\right)\\
 & \subseteq & \left( r_A^{2}r_{x}^{2k}\right)\left(r_{A}r_{x}^{k}\right)\dot{+}\cdots\dot{+} r_A^{2}\left(r_{A}r_{x}\right)\dot{+}\left( r_A^{2}r_{x}\right)r_{A}\dot{+} r_A^{2}r_{A}\\
 & \subseteq &  r_A^{3}r_{x}^{3k}\dot{+}\cdots\dot{+} r_A^{3}r_{x}^{2k}\dot{+}\cdots\dot{+} r_A^{3}r_{x}\dot{+} r_A^{3}
\end{array}$\\
and 
set by hypothesis that
we have 

$E_k^{p-1}\subseteq r_A^{p-1}r_{x}^{\left(p-1\right)k}\dot{+}\cdots\dot{+} r_A^{p-1}r_{x}\dot{+} r_A^{p-1}.$

And so we get
\[\begin{array}{lll}
E_k^p&=&
\left(r_{A}r_{x}^{k}\dot{+}\cdots\dot{+}r_{A}\right)^{p}\\
 & = & \left(r_{A}r_{x}^{k}\dot{+}\cdots\dot{+}r_{A}\right)^{p-1}\left(r_{A}r_{x}^{k}\dot{+}\cdots\dot{+}r_{A}\right)\\
 & \subseteq & \left( r_A^{p-1}r_{x}^{\left(p-1\right)k}\dot{+}\cdots\right.%
\\
 &    & \hspace{1.6cm}\left.%
\dot{+} r_A^{p-1}r_{x}\dot{+} r_A^{p-1}\right)%
\left(r_{A}r_{x}^{k}\dot{+}\cdots\dot{+}r_{A}^{\hphantom{b}}\right)\\
 & \subseteq & \left( r_A^{p-1}r_{x}^{\left(p-1\right)k}\right)\left(r_{A}r_{x}^{k}\right)\dot{+}\cdots\dot{+} r_A^{p-1}\left(r_{A}r_{x}\right)\\
&  & \hphantom{berecomeemocereb}\dot{+}\left( r_A^{p-1}r_{x}\right)r_{A}\dot{+} r_A^{p-1}r_{A}\\
 & \subseteq &  r_A^{p}r_{x}^{pk}\dot{+}\cdots\dot{+} r_A^{p}r_{x}^{2k}\dot{+}\cdots\dot{+} r_A^{p}r_{x}\dot{+} r_A^{p}
\end{array}\]
Proof is then done.
\end{proof}

\begin{lemma}\label{coefbin} 
Let $L$ be a Leibniz algebra and $\left(l,r,V\right)$
a representation of $L$. Let $A$ be a subspace of the vector space
$L$ and $x$ in the normalizer $n_{L}(A)$ of $A$. Let $m$ be a non negative integer.
Then for all $\left(\lambda,a\right)\in F\times A$,
\[
f_m=\left(r_{a+\lambda x}\right)^{m}-\sum_{k=0}^{m}\left(\hspace{-5pt}\begin{array}{c}
m\\
k
\end{array}\hspace{-5pt}\right)\lambda^{k} r_{a}^{m-k} r_x^{k}\in r_{A}r_{x}^{m}\dot{+}\cdots\dot{+}r_{A}.
\]
\end{lemma}

\begin{proof} By induction:

$
f_{1}=\left(r_{a+\lambda x}\right)^{1}-\ds\sum_{k=0}^{1}\left(\hspace{-5pt}\begin{array}{c}
1\\
k
\end{array}\hspace{-5pt}\right)\lambda^{k}r_{a}^{1-k}r_x^k
$\\
$
\hphantom{aafal}=r_{a+\lambda x}-\left(r_{a}+\lambda r_{x}\right)=0\in r_{A}r_{x}\dot{+}r_{A}.
$

And 
if by hypoyhesis
we have:  

$
f_{m}=\left(r_{a+\lambda x}\right)^{m}-\ds\sum_{k=0}^{m}\left(\hspace{-5pt}\begin{array}{c}
m\\
k
\end{array}\hspace{-5pt}\right)\lambda^{k} r_{a}^{m-k} r_x^{k}\in r_{A}r_{x}^{m}\dot{+}\cdots\dot{+}r_{A}.
$\\

Then we got:

$\ds f_{m+1}=\left(r_{a+\lambda x}\right)^{m+1}-{\displaystyle \sum_{k=0}^{m+1}}\left(\hspace{-5pt}\begin{array}{c}
m+1\\
k
\end{array}\hspace{-5pt}\right)\lambda^{k} r_a^{m-k+1} r_x^{k}$\\
$\ds f_{m+1}=\left(r_{a}+\lambda r_{x}\right)^{m+1}-{\displaystyle \sum_{k=0}^{m+1}}\left(\hspace{-5pt}\begin{array}{c}
m+1\\
k
\end{array}\hspace{-5pt}\right)\lambda^{k} r_a^{m-k+1} r_x^{k}$\\
$\ds\hphantom{f_{m+1}}=\left(r_{a}+\lambda r_{x}\right)^{m}\left(r_{a}+\lambda r_{x}\right)-{\displaystyle \sum_{k=0}^{m+1}}\left(\hspace{-5pt}\begin{array}{c}
m+1\\
k
\end{array}\hspace{-5pt}\right)\lambda^{k} r_a^{m-k+1} r_x^{k}$\\
 $\ds\hphantom{f_{m+1}}=\left(\sum_{k=0}^{m}\left(\hspace{-5pt}\begin{array}{c}
m\\
k
\end{array}\hspace{-5pt}\right)\lambda^{k} r_{a}^{m-k} r_x^{k}+f_{m}\right)\left(r_{a}+\lambda r_{x}\right)$\\
 $\ds\hphantom{aaaaaaaaaaaaaaaaaaa}-{\displaystyle \sum_{k=0}^{m+1}}\left(\hspace{-5pt}\begin{array}{c}
m+1\\
k
\end{array}\hspace{-5pt}\right)\lambda^{k} r_a^{m-k+1} r_x^{k}$\\
$\ds\hphantom{f_{m+1}}=\sum_{k=0}^{m}\left(\hspace{-5pt}\begin{array}{c}
m\\
k
\end{array}\hspace{-5pt}\right)\lambda^{k} r_{a}^{m-k} r_x^{k}r_{a}+f_{m}r_{a}$\\
 $\ds\hphantom{uuucomeu}+\sum_{k=0}^{m}\left(\hspace{-5pt}\begin{array}{c}
m\\
k
\end{array}\hspace{-5pt}\right)\lambda^{k+1} r_{a}^{m-k} r_x^{k+1}+\lambda f_{m}r_{x}$\\
 $\ds\hphantom{uuucomeuuuuu}-\sum_{k=0}^{m+1}\left(\hspace{-5pt}\begin{array}{c}
m+1\\
k
\end{array}\hspace{-5pt}\right)\lambda^{k} r_a^{m-k+1} r_x^{k}$\\

Then we have

$\ds f_{m+1}=\sum_{k=0}^{m}\left(\hspace{-5pt}\begin{array}{c}
m\\
k
\end{array}\hspace{-5pt}\right)\lambda^{k} r_{a}^{m-k}\left(r_{x}^{k}r_{a}\right)+f_{m}r_{a}$\\
 $\ds\hphantom{uuucomeuuuuu}+\sum_{k=0}^{m}\left(\hspace{-5pt}\begin{array}{c}
m\\
k
\end{array}\hspace{-5pt}\right)\lambda^{k+1} r_{a}^{m-k} r_x^{k+1}+\lambda f_{m}r_{x}$\\
$\ds\hphantom{uuucomeu}-\sum_{k=0}^{m+1}\left(\hspace{-5pt}\begin{array}{c}
m+1\\
k
\end{array}\hspace{-5pt}\right)\lambda^{k} r_a^{m-k+1} r_x^{k}$\\

Since $r_{x}^{k}r_{a}=r_{a}r_{x}^{k}+\beta_{k}$ we get

$\ds f_{m+1}=\sum_{k=0}^{m}\left(\hspace{-5pt}\begin{array}{c}
m\\
k
\end{array}\hspace{-5pt}\right)\lambda^{k} r_{a}^{m-k}\left(r_{a}r_{x}^{k}+\beta_{k}\right)+f_{m}r_{a}$\\
 $\ds\hphantom{uuucomeuuuuu}+\sum_{k=0}^{m}\left(\hspace{-5pt}\begin{array}{c}
m\\
k
\end{array}\hspace{-5pt}\right)\lambda^{k+1} r_{a}^{m-k} r_x^{k+1}+\lambda f_{m}r_{x}$\\
 $\ds\hphantom{uuucomeuuuuu}-\sum_{k=0}^{m+1}\left(\hspace{-5pt}\begin{array}{c}
m+1\\
k
\end{array}\hspace{-5pt}\right)\lambda^{k} r_a^{m-k+1} r_x^{k}$\\

$\ds f_{m+1}=\sum_{k=0}^{m}\left(\hspace{-5pt}\begin{array}{c}
m\\
k
\end{array}\hspace{-5pt}\right)\lambda^{k} r_a^{m-k+1} r_x^{k}+\sum_{k=0}^{m}\left(\hspace{-5pt}\begin{array}{c}
m\\
k
\end{array}\hspace{-5pt}\right)\lambda^{k} r_{a}^{m-k}\beta_{k}+f_{m}r_{a}$\\
 $\ds\hphantom{uuucomeuuuuu}+\sum_{k=0}^{m}\left(\hspace{-5pt}\begin{array}{c}
m\\
k
\end{array}\hspace{-5pt}\right)\lambda^{k+1} r_{a}^{m-k} r_x^{k+1}+\lambda f_{m}r_{x})$\\
 $\ds\hphantom{uuucomeuuuuu}-\sum_{k=0}^{m+1}\left(\hspace{-5pt}\begin{array}{c}
m+1\\
k
\end{array}\hspace{-5pt}\right)\lambda^{k} r_a^{m-k+1} r_x^{k}$\\
$\ds\hphantom{f_{m+1}}=r_{a}^{m+1}+\sum_{k=1}^{m}\left(\hspace{-5pt}\begin{array}{c}
m\\
k
\end{array}\hspace{-5pt}\right)\lambda^{k} r_a^{m-k+1} r_x^{k}+\sum_{k=0}^{m}\left(\hspace{-5pt}\begin{array}{c}
m\\
k
\end{array}\hspace{-5pt}\right)\lambda^{k} r_{a}^{m-k}\beta_{k}+f_{m}r_{a}$\\
 $\ds\hphantom{uuucomeuuuuu}+\lambda^{m+1}r_{x}^{m+1}+\sum_{k=0}^{m-1}\left(\hspace{-5pt}%
\begin{array}{c}
m\\
k
\end{array}\hspace{-5pt}\right)\lambda^{k+1} r_{a}^{m-k} r_x^{k+1}+\lambda f_{m}r_{x}$\\
 $\ds\hphantom{uuucomeuuuuu}-\sum_{k=0}^{m+1}\left(\hspace{-5pt}\begin{array}{c}
m+1\\
k
\end{array}\hspace{-5pt}\right)\lambda^{k} r_a^{m-k+1} r_x^{k}$\\
$\ds\hphantom{f_{m+1}}=r_{a}^{m+1}+\sum_{j=1}^{m}\left(\hspace{-5pt}\begin{array}{c}
m\\
j
\end{array}\hspace{-5pt}\right)\lambda^{k} r_a^{m-j+1} r_x^{j}+\sum_{k=0}^{m}\left(\hspace{-5pt}\begin{array}{c}
m\\
k
\end{array}\hspace{-5pt}\right)\lambda^{k} r_{a}^{m-k}\beta_{k}+f_{m}r_{a}$\\
 $\ds\hphantom{uuucomeuuuuu}+\lambda^{m+1}r_{x}^{m+1}+\sum_{j=1}^{m}\left(\hspace{-5pt}\begin{array}{c}
m\\
j-1
\end{array}\hspace{-5pt}\right)\lambda^{j} r_a^{m-j+1} r_x^{j}+\lambda f_{m}r_{x}$\\
 $\ds\hphantom{uuucomeuuuuu}-\sum_{k=0}^{m+1}\left(\hspace{-5pt}\begin{array}{c}
m+1\\
k
\end{array}\hspace{-5pt}\right)\lambda^{k} r_a^{m-k+1} r_x^{k}$\\
$\ds\hphantom{f_{m+1}}=r_{a}^{m+1}+\sum_{j=1}^{m}\left(\hspace{-5pt}\begin{array}{c}
m\\
j
\end{array}\hspace{-5pt}\right)\lambda^{k} r_a^{m-j+1} r_x^{j}$\\
$\ds\hphantom{uuucomeuuuuu}+\sum_{j=1}^{m}\left(\hspace{-5pt}\begin{array}{c}
m\\
j-1
\end{array}\hspace{-5pt}\right)\lambda^{j} r_a^{m-j+1} r_x^{j}+\lambda^{m+1}r_{x}^{m+1}$\\
 $\ds\hphantom{uuucomeuuuuu}-\left(r_{a}^{m+1}+\sum_{k=1}^{m}\left(\hspace{-5pt}\begin{array}{c}
m+1\\
k
\end{array}\hspace{-5pt}\right)\lambda^{k} r_a^{m-k+1} r_x^{k}+\lambda^{m+1}r_{x}^{m+1}\right)$\\
$\ds\hphantom{uuucomeuuuuu}+\sum_{k=0}^{m}\left(\hspace{-5pt}\begin{array}{c}
m\\
k
\end{array}\hspace{-5pt}\right)\lambda^{k} r_{a}^{m-k}\beta_{k}+f_{m}r_{a}+\lambda f_{m}r_{x}$\\

Finally we have

$\ds\hphantom{f_{m+1}}=
\sum_{k=0}^{m}\left(\hspace{-5pt}\begin{array}{c}
m\\
k
\end{array}\hspace{-5pt}\right)\lambda^{k} r_{a}^{m-k}\beta_{k}
+f_{m}r_{a}+\lambda f_{m}r_{x}$\\
$\ds\hphantom{f_{m+1}}\in\sum_{k=0}^{m}\left(\hspace{-5pt}\begin{array}{c}
m\\
k
\end{array}\hspace{-5pt}\right)\lambda^{k}\left(r_{A}\right)^{m-k}\left(r_{A}r_{x}^{m}\dot{+}\cdots\dot{+}r_{A}\right)$\\
$\ds\hphantom{fmmmmmmaaaaaaa}\dot{+}\left(r_{A}r_{x}^{m}\dot{+}
\cdots\dot{+}r_{A}\right)r_{A}\dot{+}\lambda\left(r_{A}r_{x}^{m}\dot{+}\cdots\dot{+}r_{A}\right)r_{x}$\\
$\ds\hphantom{f_{m+1}}\in r_{A}r_{x}^{m+1}\dot{+}\cdots\dot{+}\cdots\dot{+}r_{A}r_{x}\dot{+}r_{A}$.
\end{proof}

\begin{definition}

Call $x\in End(V)$ semisimple if the roots of its minimum polynomial
over $F$ are all distinct, or equivalently, if $x$ is diagonalizable.

\end{definition} 

\begin{remark}\label{sumetrest}
\begin{description}
\item[i)] Two commuting semisimple endomorphisms are simultaneously diagonalizable,
so their sum and difference are both semisimple.
\item[ii)] If $x$ is semisimple and $x$ leaves a subspace $W$ invariant,
then the restriction of $x$ to $W$ denoted by $x_{|W}$ is semisimple. 
\end{description}
\end{remark} 

\begin{definition}

Call $x\in L$ ad-semisimple (respectively Ad-semisimple) if the endomorphisms
$ad_{x}$ is semisimple (respectively $Ad_{x}$ is semisimple).

Call $x\in L$ ad-nilpotent (respectively Ad-nilpotent) if the endomorphisms
$ad_{x}$ is nilpotent (respectively $Ad_{x}$ is nilpotent).

\end{definition}

\begin{lemma}\label{tras}
 Let $V=V_1\oplus V_2$ be a direct sum of two vector spaces $V_1,V_2$, an non negative integer $p$ and $\sigma$ an endomorphism of $V$ shuch that $\sigma^p(V)\subseteq V_1$, then the trace of $\sigma$  denoted by $tr(\sigma)=tr(\sigma_{|V_1})$, where $\sigma_{|V_1}$ is the restriction of $\sigma$ to $V_1$.
\end{lemma}

\begin{proof}
Since we have an algebraically closed field, we can find a basis $\{v_1,\cdots,v_m,\cdots, v_n\}$ of $V$ whith $\{v_1,\cdots,v_m,\}$ is a basis of $V_1$ and scalars $\lambda_1,\cdots,\lambda_n$ shuch that the matrix of $\sigma$ in this basis is 
$$
 N_{0k}=\left(\begin{array}{ccccc}
\lambda_{1} & a_{1,2} & a_{1,3} & \cdots & a_{1,n}\\
0 & \lambda_{2} & a_{2,3} & \cdots & a_{2,n}\\
\vdots & 0 & \ddots & \ddots & \vdots\\
\vdots & \vdots & \ddots & \lambda_{n-1} & a_{n-1,n}\\
0 & 0 & \cdots & 0 & \lambda_{n}
\end{array}\right)
$$
For $m+1\leq i\leq n$, we have a vector $0\neq v_i\in V_2$ shuch that $\sigma(v_i)=\lambda_iv_i$. \\ Then 
$\sigma^p(v_i)=\lambda_i^pv_i\in V_2\cap V_1=\{0\}$. So $\lambda_i=0$ for $m+1\leq i\leq n$, and $$tr(\sigma)=\ds\sum_{j=1}^n\lambda_j=\ds\sum_{j=1}^m\lambda_j=tr(\sigma_{|V_1})$$.
\end{proof}

\section{Radical and Nilradical.}\label{sec3} 

The proof of following proposition  can be found in  \cite{Hymph}.
\begin{proposition}\label{prop151} 
Let $\mathfrak{W}$ be a Lie subalgebra of $End_F(V)$ where $V$
is an $F$-vector space. Then $\mathfrak{W}$ is solvable if and only
if $tr(x\circ y)=0$ for all $x\in\mathfrak{W}$ and $y\in[\mathfrak{W},\mathfrak{W}]$. 
\end{proposition}

\begin{theorem}\cite[Theorem 3.7]{Albeverio+} 
Let $L$ be a  Leibniz algebra. Then $L$
is solvable if and only if for all $x$ in $L$ and all $y$ in $\left[L,L\right]$,$tr\left(ad_{x}\circ ad_{y}\right)=0$. 
\end{theorem}


If $\imath$ is an ideal of $L$ and $L/\imath$ is solvable 
(respectively
nilpotent)
, then $D^{(n)}(L/\imath)=0$ 
(respectively
$\left(L/\imath\right)^{n}=0$) 
 implies that $D^{(n)}(L)\subset\imath$
(respectively $L^{n}\subset\imath$ nilpotent).  
If
$\imath$ itself is solvable with $D^{(m)}(\imath)=0$ 
(respectively nilpotent with $\imath^{m}=0$), 
then $D^{(m+n)}(L)=0$  
(respectively
$L^{m+n}=0$).\\
So we have proved: 

\begin{proposition}\label{quot12}
If $\imath\subset L$ is an ideal, and both $\imath$ and $L/\imath$
are solvable (respectively nilpotent), so is $L$
solvable (respectively nilpotent). 
\end{proposition} 

If $\imath$ and $\jmath$ are solvable ideals, then $(\imath+\jmath)/\jmath\equiv\imath/(\imath\cap\jmath)$
is solvable, being the homomorphic image of a solvable algebra. So,
by the previous propositio, we have the 

\begin{proposition} 
If $\imath$ and $\jmath$ are solvable ideals (respectively
nilpotent ideals) in $L$ so $\imath+\jmath$ is solvable (respectively
nilpotent). In particular, every Leibniz algebra $L$ has a largest
solvable ideal which contains all other solvable ideals and a largest
nilpotent ideal which contains all other nilpotent ideals.\\ 
The largest solvable one is denoted by $Rad\left(L\right)$.\\  %

The largest nilpotent one is denoted by $Nil\left(L\right)$. %
\end{proposition} 
\begin{remark}
 Note that $\Ess(L)\subseteq Nil\left(L\right)\subseteq Rad\left(L\right)$.
\end{remark}

\section{The ideal $\{Rad(L),L\}$.}\label{sec4} 

Let us denote the subspace $\left[Rad(L),L\right]\dot{+}\left[L,Rad(L)\right]$ by $\left\{Rad(L),L\right\}$.

\begin{lemma} Let $L$ be a Leibniz algebra and $\left(l,r,V\right)$
a representation of $L$. Let $A$  be a subspace of $L$ for which
there exists an integer $n\in\mathbb{N}^{*}$ with $ r_A^{n}=\left\{ 0\right\}$
and let $x$ be in $n_{L}(A)$ such that $r_{x}$ is nilpotent. Then there
exists an integer $N\in\mathbb{N}^{*}$ with $\left(r_{A+Fx}\right)^{N}=\left\{ 0\right\}$.
\end{lemma} 

\begin{proof} 

Let us notice that for any  non negative integer $p$ we have\\
$\left(r_{a+\lambda x}\right)^{p}=\ds\sum_{k=0}^{p}\left(\hspace{-5pt}\begin{array}{c}
p\\
k
\end{array}\hspace{-5pt}\right)\lambda^{k} r_x^{k}\left(r_{a}\right)^{p-k}+f_p$ where $f_p\in E_p=r_{A}r_{x}^{p}\dot{+}\cdots\dot{+}r_{A}.$

Let $m$ an integer with $\left(r_{x}\right)^{m}=0$. Then 
with $p=2\sup\left(m,n\right)+1>m+n$ we have that 
$\left(r_{a+\lambda x}\right)^{p}=f_p\in E_p.$
And so 
\[
\begin{array}{lll}
\left[\left(r_{a+\lambda x}\right)^{p}\right]^n&=&\left(f_p\right)^n=\left(r_{A}r_{x}^{p}\dot{+}\cdots\dot{+}r_{A}\right)^n\\
&\subseteq& r_A^{n}r_{x}^{np}\dot{+}\cdots\dot{+}%
 r_A^{n}r_{x}^{2p}\dot{+}\cdots
\dot{+} r_A^{n}r_{x}\dot{+} r_A^{n}
\end{array}
\]

Since $ r_A^{n}=\{0\}$, $\left(r_{a+\lambda x}\right)^{pn}=0$. So $r_{a+\lambda x}$  is nilpotent for all $a+\lambda x$ in $A\dot{+}Fx$.
By \cite[Theorem 3.2., page 41]{Schaf}  the associative algebra $r_{A\dot{+}Fx}$ 
is nilpotent algebra. So there is some integer $N\in\mathbb{N}^*$ such that $\left(r_{A+Fx}\right)^{N}=\left\{ 0\right\}$.
\end{proof} 

\begin{proposition} For any representation $\left(l,r,V\right)$
of the Leibniz algebra $L$, the restriction of $r$ to the ideal
$\left\{Rad(L),L\right\}$ is nilpotent, i.e. there exists an integer
$m\in\mathbb{N}^{*}$ with $\left(r_{\left\{Rad(L),L\right\}}\right)^{m}=\left\{ 0\right\} $.
\end{proposition} 

\begin{proof} 
According to \cite[Corollary 4.4]{bere12} 
the representation of $V$ is nilpotent
on the ideal $\left[L,L\right]$. Now let $T\subseteq\left\{Rad(L),L\right\}$
be a subspace containing $\left[Rad(L),Rad(L)\right]$, which is maximal with
respect to the property that the representation of $V$ is nilpotent
on $T$. Note that $T$ always is an ideal of $Rad\left(L\right)$,
hence in particular a subalgebra, because it contains $\left[Rad(L),Rad(L)\right]$.

Assume that $T\neq\left\{Rad(L),L\right\}$. Then there exist at least an $x$
in $Rad(L)$ and $y$ in $L$ with $\left[x,y\right]\notin T$ or $\left[y,x\right]\notin T$. 

\begin{description}
  \item[If]  $\left[x,y\right]\notin T$,
the subspace $B=Rad\left(L\right)\dot{+}Fx$ is a subalgebra of $L$, $Rad\left(L\right)$
is a solvable ideal of $B$ and $B/Rad\left(L\right)\approx F$ is
abelian. Therefore $B$ is a solvable ideal by Proposition \ref{quot12}.\\
Again we use \cite[Corollary 4.4]{bere12} 
to see that the representation of $V$ is nilpotent on $\left[B,B\right]$ and hence that $r_{\left[x,y\right]}$
is nilpotent.

Since $T\subseteq Rad(L)$ and $\left[x,y\right]\in\left[Rad\left(L\right),y\right]\subseteq Rad\left(L\right)$,
we have\\ $\left[\left[x,y\right],T\right]\subseteq\left[Rad\left(L\right),T\right]\subseteq T$ and $\left[T,\left[x,y\right]\right]\subseteq\left[T,Rad\left(L\right)\right]\subseteq T$.\\
Finally the preceding lemma  show that the representation of $V$
is nilpotent on the subspace $T\oplus F\left[x,y\right]$. This contradicts
the maximality of $T$.

  \item[If]  $\left[y,x\right]\notin T$,
the subspace $B=Rad\left(L\right)\dot{+}Fx$ is a subalgebra of $L$, $Rad\left(L\right)$
is a solvable ideal of $B$ and $B/Rad\left(L\right)\approx F$ is
abelian. Therefore $B$ is a solvable ideal by Proposition \ref{quot12}.\\
Again we use \cite[Corollary 4.4]{bere12}  
to see that the representation on $V$ is nilpotent on $\left[B,B\right]$ and hence that $r_{\left[y,x\right]}$
is nilpotent.

Since $T\subseteq Rad(L)$ and $\left[y,x\right]\in\left[y,Rad\left(L\right)\right]\subseteq Rad\left(L\right)$,
we have\\ $\left[\left[y,x\right],T\right]\subseteq\left[Rad\left(L\right),T\right]\subseteq T$ and $\left[\left[y,x\right],T\right]\subseteq\left[Rad\left(L\right),T\right]\subseteq T$.\\
Finally the preceding lemma  show that the representation of $V$
is nilpotent on the subspace $T\oplus F\left[x,y\right]$. This contradicts
the maximality of $T$.
 \end{description}
 We conclude that $T$ must be equal to $\left\{Rad(L),L\right\}$,
so the representation of $V$ is nilpotent on $\left\{Rad(L),L\right\}$.
\end{proof} 

Applying the precedent proposition to the adjoint representation $\left(Ad,ad,L\right)$
of the Leibniz algebra $L$ and using Engel's Theorem \cite{Barnes1}, we get the:

\begin{corollary}The ideal $\left\{Rad(L),L\right\}$ is nilpotent.
In particular, $x$ is ad-nilpotent for every $x$ in $\left\{Rad(L),L\right\}$.
\end{corollary} 


\begin{corollary} 
Let $L$ be a Leibniz algebra and $D$ a derivation
of $L$.\\ 
Then $D\left(Rad(L)\right)\subseteq Nil(L)$. In particular
$Nil(L)$ is a characteristical ideal.
\end{corollary} 

\begin{proof}
For a derivation $D$ of $L$, define the Leibniz
algebra $\tilde{L}=L\times\!|_{D}F$ with the bracket 
$
\left[\left(x,t\right),\left(y,l\right)\right]=\left(lD\left(x\right)-tD\left(y\right)+\left[x,y\right],0\right).
$
Then, 
$
(D\left(Rad(L)\right),0)=\left[\left(Rad(L),0\right)\left(0,1\right)\right]\subseteq (L,0)\cap\left[Rad(\tilde{L}),\tilde{L}\right]\subseteq \tilde{L}\cap Nil\left(\tilde{L}\right)\subseteq Nil\left(\tilde{L}\right)=%
(nil\left(L\right),0).
$
So $D\left(Rad(L)\right)\subseteq Nil\left(L\right)$.
\end{proof}

\section{Main theorem.}\label{sec5}

We deal in this section with Leibniz algebras which sastify equation
\[
\forall x,y\in L,tr\left(ad_{x}\circ ad_y\right)_{|\Ess\left(L\right)}= 0\]

Call such  Leibniz algebras: Killing Leibniz Algebras.

A bilinear form $(-,-):L\times L\longrightarrow F$ is called invariant
if 
\[
([x,y],z)+(y,[x,z])=0
\]
 for all $x,y,z$ in $L$. Notice that if $(-,-)$ is an invariant
form, and $\imath$ is an ideal, then its orthogonal $\imath^{\perp}$
is again an ideal. 

One way of producing invariant forms is from representations: if $(l,r,V)$
is a representation of $L$, then
\[
(x,y)_{r}=tr(r_{x}\circ r_{y})
\]
 is invariant. Indeed, 
\[
([x,y],z)_{r}+(y,[x,z])_{r}
\]
\[
=tr\left((r_{y}\circ r_{x}-r_{x}\circ r_{y})\circ r_{z}+r_{y}\circ(r_{z}\circ r_{x}-r_{x}\circ r_{z})\right)
\]
\[
=tr\left(\left(r_{y}\circ r_{z}\right)\circ r_{x}-r_{x}\circ\left(r_{y}\circ r_{z}\right)\right)
=0\]
 

In particular, if we take $l=Ad,\, r=ad$, $V=L$ the corresponding
bilinear form is called the Killing form and will be denoted by $\mathfrak{K}=(-,-)_\mathfrak{K}$. 

\begin{remark}
for all $x$ in $\Ess(L)$, $y$, $z$ in $L$ we have:\\
$(ad_{x}\circ ad_{y})(z)=(ad_{x})([z,y])=[[z,y],x]=0.$\\
Then $ad_{x}\circ ad_{y}\equiv0 $  and
$(x,y)_{\mathfrak{K}}=tr(ad_{x}\circ ad_{y})=0$, so $\Ess(L)\subseteq\ker(\mathfrak{K})$. 
\end{remark}

\begin{theorem}\label{sol1} 
Let $L$ be a leibniz algebra of a class Killing Leibniz Algebras and  $\ker(\mathfrak{K})$ the kernel of its Killing form.\\
 $\ker(\mathfrak{K})=\Ess(L)$  if and only if $L$ is semisimple.
\end{theorem}

\begin{proof} 
Suppose that $L$ is semisimple. 
Let us show that the kernel of the Killing form is $\Ess(g)$.\\
So let $\mathfrak{W}=L^{\perp}=\left\{ x\in L,\, tr\left(ad_{x}\circ ad_{y}\right)=0\textnormal{ for all }y\in L\right\} $.
If $x\in\mathfrak{W}$, $y,z\in L$ then

$tr\left(ad_{[x,z]}\circ ad_{y}\right)=tr\left(ad_{x}\circ ad_{z}\circ ad_{y}-ad_{z}\circ ad_{x}\circ ad_{y}\right)=tr\left(ad_{x}\circ\left(ad_{z}\circ ad_{y}-ad_{y}\circ ad_{z}\right)\right)$\\
$\hphantom{tr\left(ad_{[x,z]}\circ ad_{y}\right)}=tr\left(ad_{x}\circ ad_{\left[z,y\right]}\right)=0,$\par 
%
%
And so on, we have also $tr\left(ad_{[z,x]}\circ ad_{y}\right)=0$.

So $\mathfrak{W}$ is an ideal and clearly $\Ess(L)\subseteq\mathfrak{W}$.\\
$ad_{\mathfrak{W}}$ is a solvable a Lie subalgebra of $End(V)$ by
Cartan's criterion. Thanks to Proposition \ref{prop151}, $\mathfrak{W}$ is
solvable and hence $\mathfrak{W}=Rad(L)=\Ess\left(L\right)$.



Conversely,\\
suppose $L$ is not semisimple and so has a solvable ideal such that $a\supsetneq\Ess(L)\supseteq a^{2}$ by Remark \ref{ari}.
%
Let us show that $(x,y)_{{K}}=0$ for all $x$ in $a$,
$y$ in $L$ and then $a\subset\ker(\mathfrak{K})$.\\ 
Let $\sigma=ad_{x}\circ ad_{y}$. 


By assumption %
$tr(\sigma_{|\Ess(L)})=0$.

And since $\sigma$ maps $L$ to $a$, $a$
to $a^{2}$ and  
$a^{2}\subseteq\Ess\left(L\right)$, we have that 
\[
\sigma^{2}\left(L\right)\subseteq\sigma\left(a\right)\subseteq a^{2}\subseteq%
\Ess\left(L\right).
\]
Write $L=\Ess\left(L\right)\oplus L_2$. Then we have by Lemma \ref{tras}, that 
$tr(\sigma)=tr(\sigma_{|\Ess(L)})=0$. Hence if $L$ is
not semisimple then the kernel of its Killing form satisfies $\Ess(L)\subsetneq\ker({\mathfrak{K}})$.
\end{proof}

\begin{remark}
 I. Demir et al. give another proof of this theorem :
Leibniz algebra is semisimple implies the Killing form is non degenerate. (see \cite[Theorem 5.8]{demirs}).
\end{remark}

\section{Conclusion}
Let us cite an example of Leibniz algebra which is solvable and the kernel of it's Killing form is $\Ess(L)$.
\begin{example}\label{examp}\cite{Masons} %

Let $L=\mathbb{C}x+\mathbb{C}y$ be the two dimensional complex Leibniz
algebra which generators satisfy $\left[x,x\right]=\left[y,y\right]=\left[y,x\right]=0;\left[x,y\right]=x$.
\end{example}

Let us find the kernel of the Killing form of the non lie leibniz algebra $L=Fx\oplus Fy$ defined in Example \ref{examp}. 
Let $a=a_{11}x+a_{12}y$ and $b=a_{21}x+a_{22}y$ be two elements of algebra. The matrix of the endomorphism $ad_a$  is 
$\begin{pmatrix}
                                                                                                                                                                                                                                 a_{12}&0\\
0&0
                                                                                                                                                                                                                           \end{pmatrix}$
 and
the matrix of the endomorphism $ad_b$  is  $\begin{pmatrix}
               a_{22}&0\\
		0&0
              \end{pmatrix}$.

Then the Killing form is defined by $(a,b)_{\mathfrak{K}}=a_{12}a_{22}$ for all $a,b$ in $L$.

Since $\Ess(L)=\{0\}$ for any Lie algebra; Lie algebras are Killing Leibniz algebras   and 
the Theorem \ref{sol1} is knowned for Lie algebras (cf. \cite{Hymph}).\\ 
"Left central Leibniz" are also Killing Leibniz algebras.\\
 Example \ref{examp}  is an algebra not in a class of 
Killing Leibniz algebras.

We claim that
 
\textbf{Claim:}
The class of Leibniz algebras of type W-L-A is a widest class wich satisfies  Theorem \ref{sol1}.

In \cite{Masons}, the authors call an algebra that is both a left and right Leibniz algebra a symmetric Leibniz algebra. %
they call L a left central Leibniz algebra if it is a left Leibniz algebra that also satisfies $[[a,a],b] = 0, a\in L, b\in L$.
There is a hierarchy of algebras
\\
$\{left Leibniz\} \supsetneq \{left central Leibniz\} \supsetneq \{symmetric Leibniz\} \supsetneq \{Lie\}$.

We call a right central Leibniz algebra if it is a right Leibniz algebra that also satisfies $[b,[a,a]] = 0, a\in L, b\in L$ ; and there is a hierarchy of algebras
\\
$\{right Leibniz\} \supsetneq \{right central Leibniz\} \supsetneq \{symmetric Leibniz\} \supsetneq \{Lie\}$.

So we can complete the  hierarchy of Leibniz algebras  as 
\\
$\{left Leibniz\} \supsetneq \{left Killing Leibniz\} \supsetneq \{left central Leibniz\} \supsetneq \{symmetric Leibniz\} \supsetneq \{Lie\}$.
\\
and
\\
$\{right Leibniz\} \supsetneq \{right Killing Leibniz\}  \supsetneq \{right central Leibniz\} \supsetneq \{symmetric Leibniz\} \supsetneq \{Lie\}$.

\textbf{Questions:}
\begin{itemize}
  \item Can we prove the Weyl's theorem on complete reducibility for Killing Leibniz Algebras?
  \item In \cite{Masons}, the authors show that "left central Leibniz algebras" are one class of algebras which satisfies version of the Malcev theorem. Is it also true for 
 Killing Leibniz Algebras?
\end{itemize}

\LastPageEnding

\end{document}